\documentclass[a4paper,reqno]{amsart}
\usepackage{amsmath, amssymb, eucal, amscd, amstext, verbatim, enumerate, color}
\usepackage{microtype}
\usepackage{amssymb,latexsym}
\usepackage{mathtools}
\usepackage{amsfonts}
\usepackage{mathrsfs}
\usepackage{blindtext}
\usepackage{latexsym}
\usepackage[colorlinks]{hyperref}
\usepackage{graphicx}
\usepackage[all]{xy}
\usepackage{xcolor}
\usepackage{layout}
\usepackage{bbm}
\usepackage[pagewise]{lineno}

\theoremstyle{plain}
\newtheorem{thm}{Theorem}[section]
\newtheorem{lem}[thm]{Lemma}
\newtheorem{prop}[thm]{Proposition}
\newtheorem{cor}[thm]{Corollary}

\theoremstyle{definition}
\newtheorem{defn}[thm]{Definition}
\newtheorem{ex}[thm]{Example}

\newtheorem{rem}{Remark}
\numberwithin{equation}{section}

\newcommand{\BN}{\mathbb{N}}
\newcommand{\BC}{\mathbb{C}}
\newcommand{\CE}{\mathcal{E}}
\newcommand{\BZ}{\mathbb{Z}}
\newcommand{\A}{\mathcal{A}}
\newcommand{\B}{\mathcal{B}}
\newcommand{\C}{\mathcal{C}}
\newcommand{\D}{\mathcal{D}}
\newcommand{\G}{\Gamma}
\newcommand{\M}{\mathbb M}
\newcommand{\GC}{$\Gamma$-C$^*$-algebra}
\newcommand{\CH}{\mathcal{H}}
\newcommand\ssubset{
	\mathrel{
		\mathrlap{\subset}
		\hphantom{\ll}
		\mathllap{\subset}
	}
}
\begin{document}
\title{Equivariant (co)module nuclearity of C$^*$-crossed products}

\author{Massoud Amini, Qing Meng}

\address[Massoud Amini]{Faculty of Mathematical Sciences, Tarbiat Modares University, Tehran 14115134, Iran}
\email{mamini@modares.ac.ir}
\address[Qing Meng]{School of Mathematical Sciences, Qufu Normal University, Qufu, Shandong  province 273165, China}
\email{mengqing80@163.com}

\keywords{nuclearity; C$^*$-module; reduced crossed product}

\subjclass[2010]{Primary 46L05, Secondary 46L55}

\date{\today}
\maketitle

\begin{abstract}
We define an equivariant and equicovariant versions of the notion of module nuclearity. More precisely, for a discrete group $\G$ and operator $\A$-$\G$-(co)module $\B$, $\CE$ over a $\G$-C$^*$-algebra $\A$, we define $\CE$-$\G$-nuclearity of $\B$, as an equivariant version of the notion of $\CE$-nuclearity, in which the identity map on $\B$ is required to be approximately factored through matrix algebras on $\CE$ with module structures coming both from the original module structure of $\CE$ and the $\G$-action on $\CE$. For trivial actions of $\G$, this is shown to reduce to the notion of module nuclearity, introduced and studied by the first author. As a concrete example, for a discrete group $\G$ acting amenably on a unital C$^*$-algebra $\A$, we show that the reduced crossed product $\A\rtimes_{r} \G$ is $\A$-$\G$-nuclear. Conversely, if $\A$ is a nuclear C$^*$-algebra with a $\G$-invariant state $\rho$ and $\A\rtimes_{r} \G$ is $\A$-$\G$-nuclear, then we deduce that $\G$ is amenable. We show that when $\A\rtimes_{r} \G$ is $\A$-$\G$-nuclear and $\A$ has the completely bounded approximation property (resp., is exact), then so is $\A\rtimes_{r} \G$. We prove similar results for $\A\rtimes_{r} \G$, regarded as an $\A$-$\G$-comodule.
\end{abstract}

\section{Introduction}

Approximation theory plays a central role in the modern theory of operator algebras. There are many significant approximation properties that are defined and studied for C$^*$-algebras, including nuclearity, quasidiagonality, completely bounded approximation property (CBAP), Haagerup property, and weak Haagerup property (see \cite{BO, Dong, Haag, Meng1}).

The theory of nuclear C$^*$-algebras has developed for a quite while, and is closely related to the amenability of groups (see \cite{BO, Kirc, Lan}). It is natural to try to extend nuclearity to a more general setting.
The first author \cite{Ami} introduced the notion of nuclear module morphisms for a C$^*$-algebra $\B$ which is C$^*$-module over another C$^*$-algebra $\A$ (or more generally, over an operator system) with compatible actions, and used it to define nuclearity and exactness in this subcategory of C$^*$-modules.

In this paper, we discuss the nuclearity of C$^*$-crossed products, considered as C$^*$-modules on the ambient C$^*$-algebra. It turns out that this case does not fit into the module nuclearity context of \cite{Ami}, as $\A$ is not just a C$^*$-algebra, but a \GC. This suggests that an equivariant version of module amenability is needed, and this is the main objective of the current paper. We show that for a discrete group $\G$ acting amenably on a unital C$^*$-algebra $\A$,  the reduced crossed product $\A\rtimes_{r} \G$ is $\A$-$\G$-nuclear. We find applications for other finite approximation properties such as exactness and completely bounded approximation property.

\section{Equivariant (Co)module Nuclearity}
Let $\G$ be a discrete group acting by automorphisms on a C$^*$-algebra $\A$. $\A^+$ is the positive cone of $\A$, $\A^*$ is the dual space of $\A$, $\A^*_+$ is the set of all positive linear functional on $\A$.
Let $\B, \D$ be operator right $\A$-modules and $\CE$ be an operator system and operator right $\A$-module. The first author defined in \cite{Ami} the $\CE$-nuclearity for a module map $\theta: \B \to \D$  by requiring that $\theta$ is approximately decomposed through $\M_n(\CE)$ over any given finite subset of $\B$. As it would be too restrictive to
assume that the maps implementing the approximate decomposition are module maps, a class of admissible maps was defined to make the decomposition more likely to happen.

The class of admissible maps is defined as follows. To simplify notations, we use c.p. for ``completely positive'', u.c.p. for ``unital completely positive'', c.c.p. for ``contractive completely positive'', c.i. for ``complete isometry'', and c.b. for ``completely bounded''.

\begin{defn} \cite[Definition 2.1]{Ami}\label{def-admissible1}
The class of $\A$-admissible c.p. maps between operator $\A$-modules $\B$ and $\D$ is characterized through the following set of rules:
\begin{enumerate}
	\item a c.p. module map $\varphi: \B\rightarrow \D$ is $\A$-admissible,
	
	\item if a c.p. map $\varphi: \B\rightarrow \D$ is $\A$-admissible, then so are the maps $x\mapsto \varphi(bxb^*)$ and $x\mapsto d\varphi(x)d^*$, for each $b\in \B$ and $d \in \D$,
	
	\item the c.p. maps $\varphi: \B\rightarrow \D$ of the form $\varphi(x) = \rho(x)d$, where $\rho \in \A^*_+$ and $d \in \D^+$, are $\A$-admissible,

    \item the composition, positive multiples, or finite sums of $\A$-admissible c.p. maps are $\A$-admissible.
\end{enumerate}
\end{defn}

The notion of nuclear module maps was then defined as follows.

\begin{defn} \cite[Definition 2.2]{Ami} \label{e-nuc}
A c.p. module map $\theta: \B \rightarrow \D$ is called $\CE$-nuclear if there are c.c.p. $\A$-admissible maps $\varphi_n: \B\rightarrow \M_{k(n)}(\CE)$ and $\psi_n: \M_{k(n)}(\CE)\rightarrow \D$ such that
$$\|\psi_n\circ \varphi_n(b)-b\|\rightarrow 0$$
for all $b\in \B$.
\end{defn}

A C$^*$-algebra $\B$ is called a right $\A$-C$^*$-module if $\B$ is a  Banach right $\A$-module with compatibility conditions,
$$(ab)\cdot u = a(b\cdot u), \qquad a \cdot uv = (a \cdot u) \cdot v,$$
for all $a, b \in \B$ and $u,v \in \A$.
Let $\B$ be a right $\A$-C$^*$-module, then $\B$ is called $\A$-nuclear if the identity map on $\B$ is $\A$-nuclear.

In this paper we are interested in $\A$-nuclearity of the reduced crossed product $A\rtimes_r\G$. Though there is a natural way to decompose the identity map via matrix algebras over $\A$, these matrix algebras have canonical module structures coming from the $\G$-action $\alpha$. For a finite subset $F\subseteq \G$ with cardinality $n$, a typical element of $\M_n(\A)$ is a finite sum of matrices of the form $a\otimes e_{s,t}$, for $s,t\in F$, where $e_{s,t}$ is the standard matrix unit. The module structure on $\M_n(\A)$ is now defined by $(a\otimes e_{s,t})\cdot b:=a\alpha_{t^{-1}}(b)\otimes e_{s,t}$. This module structure depends not only on $n$ but also on the choice of $F$, and so the canonical decomposition of the identity map on $\A\rtimes_r \G$ via such matrix algebras over $\A$ could not be used to verify the $\A$-nuclearity of $\A\rtimes_r \G$ in the usual module nuclearity context of \cite{Ami}.

The above difficulty could be handled using the fact that $\A$ is indeed a $\G$-C$^*$-algebra, and incorporating this fact into the definition of the module structure of the matrix algebras involved. To be more precise, we need to modify our notion not only by using $\G$-action in defining the module structure but also by replacing the usual full matrix algebras $\M_n(\A)$ by matrix algebras of the form $\M_{F}(\A)$, for finite subsets $F$ of the union of all finite products of copies of $\G$.

Consider the $n$-fold Cartesian products $\G^n$ for $n\geq 1$, and put $\G^\infty:=\bigcup_{n\geq 1} \G^n$.  Given $t=(t_1,\cdots, t_\ell)\in \G^\infty$, we use the formal notations $t^{-1}:=(t_\ell^{-1},\cdots, t_1^{-1})\in\G^\infty$  and $\bar t:=t_1\cdots t_\ell\in \G$. Also, for $r\in\G$, we define $rt:=(rt_1, t_2,\cdots, t_\ell)$, and note that $\overline{t^{-1}}=\bar t^{-1}$ and $\overline{rt}=r\bar{t}$.

For a finite subset $F\subseteq \G^\infty$, let $\M_F(\A)$ be the full matrix algebra of square matrices of size $|F|\times|F|$ with components in $\A$. We define the $\A$-module structure of $\M_F(\A)$ as follows. Each element of $\M_F(\A)$ is a finite sum of elements of the form $a\otimes e_{s,t}$, where $e_{s,t}$ is the matrix unit indexed by $s,t\in F$. Define the right module action of $\A$ on $\M_F(\A)$ by $(a\otimes e_{s,t})\cdot b:=a\alpha_{\bar t}^{-1}(b)\otimes e_{s,t}$, for $a,b\in \A$ and $s,t\in F$. Note that we have chosen to work with finite subsets of $\G^\infty$, not just finite subsets of $\G$, since we need to make sure we get full matrix algebras of any large size, even when $\G$ is finite (just imagine the case where $\G$ is the trivial group, whereas $\G^\infty$ is countable).

If $\CE$ is an operator space (resp., system) and operator right $\A$-module and $\G$ acts $\CE$ by (resp., unital) surjective complete isometries, then the right $\A$-module structure of $\M_F(\CE)$ is defined by
$(x\otimes e_{s,t})\cdot b:=x\cdot \alpha_{\bar t}^{-1}(b)\otimes e_{s,t}$, for $a\in \A$, $x\in\CE$ and $s,t\in F$, with the understanding that $\alpha_{\bar t}$ is no longer an automorphism, and just is a (resp., unital) surjective c.i. map. We refer the reader for min and max tensor products of operator systems to \cite{kptt}. In this paper, $\otimes$ always denotes the min tensor product. 

\begin{defn}
	An operator $\A$-$\G$-module $\B$ is an operator system with a $\G$-action $\alpha$ by complete isometries (i.e., each $\beta_t: \B\to \B$ is a surjective c.i. map and $\beta_t\beta_s=\beta_{ts}$, for $t,s\in\G$), and a right $\A$-module action, which is compatible with the the $\G$-action $\alpha$ on $\A$ (i.e., $\beta_t(b\cdot a)=\beta_t(b)\cdot\alpha_t(a)$, for $t\in\G, a\in\A, b\in\B$). When $\B$ is a C*-algebra, we  moreover assume that, $\beta_t(bc)=\beta_t(b)\beta_t(c)$, for $t\in\G, b,c\in\B$, i.e., we let $\G$ act on $\B$ by automorphisms.
\end{defn}

\begin{defn}
	An operator reduced $\A$-$\G$-comodule $\B$ is an operator system with a reduced $\G$-coaction $\delta_\B$  (i.e., $\delta_\B: \B\to \B\otimes C^*_r(\G)$ is a non-degenerate  c.p. map), a reduced $\G$-coaction $\delta_\A$  (i.e., $\delta_\A: \A\to \A\otimes C^*_r(\G)$ is a non-degenerate homomorphism), and a right $\A$-module structure which is compatible with the $\G$-coactions (i.e., $\delta_\B(b\cdot a)=\delta_\B(b)\cdot\delta_\A(a)$, for $a\in\A, b\in\B$). When $\B$ is a C*-algebra, we further assume  that $\delta_\B$ is a homomorphism. The operator full $\A$-$\G$-comodules are defined similarly using full $\G$-coactions $\delta_\B: \B\to \B\otimes C^*(\G)$.
\end{defn}

Note that in the above definition, following \cite{q}, in both cases we use the minimal tensor products (for a version which uses maximal tensor products, see \cite{r}).

\begin{defn}
	Let $\B$ and $\D$ be operator $\A$-$\G$-modules with $\G$-actions $\alpha^\B$ and $\alpha^\D$. A map $\theta: \B \rightarrow \D$ is called a (module) $\G$-map if ($\theta$ is an $\A$-module map and) we have, $$\theta(\alpha^\B_s (b))=\alpha^\D_s(\theta(b)) \ \ \ (b\in\B, s\in\G).$$
Similarly, when $\B$ and $\D$ are operator $\A$-$\G$-comodules with full $\G$-coactions $\delta_\B:\B\to\B\otimes C^*(\G)$  and $\delta_\D:\D\to\D\otimes C^*(\G)$, a c.b. map $\theta: \B \rightarrow \D$ is called a (module) $\G$-comap if ($\theta$ is an $\A$-module map and) we have, $$\delta_\D\theta(b)=(\theta\otimes {\rm id})\delta_\B(b) \ \ \ (b\in\B).$$
The same notion is defined for reduced $\G$-coactions with $C^*(\G)$ replaced by $C^*_r(\G)$.   				
\end{defn}

\begin{defn}\label{family}
Let $\B$ and $\D$ be operator $\G$-modules with $\G$-actions $\alpha^\B$ and $\alpha^\D$.

$(i)$ A $\G$-family of operator subsystems of $\B$ is a family $\{\B_i\}_{i\in I}$ endowed with an action of $\G$ on $I$ (as a discrete space) such that  $\alpha^\B_t(\B_i)\subseteq \B_{t\cdot i}$, for $i\in I, t\in\G$.

$(ii)$ A $\G$-family of maps on a $\G$-family $\{\B_i\}_{i\in I}$ of operator subsystems of $\B$, is a family $\{\psi_i: \B_i\to \D\}_{i\in I}$ of c.p. maps such that
	$\psi_{t\cdot i}(\alpha^\B_t(b))=\alpha^\D_t\big(\psi_i(b)\big)$, for $i\in I, b\in \B_i, t\in\G$.
\end{defn}

\begin{defn}\label{cofamily}
	Let $\B$ and $\D$ be operator $\G$-comodules with $\G$-coactions $\delta_\B$ and $\delta_\D$.
	
	$(i)$ A $\G$-cofamily of operator subsystems of $\B$ is a family $\{\B_i\}_{i\in I}$ endowed with a map $\delta: I\to  I$ such that $\delta_\B(\B_i)\subseteq \B_{\delta(i)}\otimes C^*_r(\G)$, for $i\in I$.
	
	$(ii)$ A $\G$-cofamily of maps on a $\G$-family $\{\B_i\}_{i\in I}$ of operator subsystems of $\B$, is a family $\{\psi_i: \B_i\to \D\}_{i\in I}$ of c.p. maps such that $$(\psi_{\delta(i)}\otimes{\rm id}_{C^*_r(\G)})(\delta_\B(b))=\delta_\D\big(\psi_i(b)\big),\ \  (i\in I, b\in \B_i).$$
\end{defn}

The point of the above definitions is that the subsystems $\B_i$ are not assumed to be $\G$-(co)submodules of $\B$. When this is the case, we may let $\G$ act trivially on $I$ (let $\delta:={\rm id}_I$) and have a $\G$-(co)family consisting of $\G$-(co)maps.

\begin{defn}\label{def-admissible}
	The class of $\G$-admissible families of c.p. maps is characterized through the following set of rules:
	\begin{enumerate}
		\item[($i$)] a $\G$-family is $\G$-admissible,
		\item[($ii$)] if a family $\{\psi_i: \B_i\to \D\}_{i\in I}$ is $\G$-admissible, it remains $\G$-admissible if one of the following maps is added to the family:

		\noindent\hspace{.8cm} (1) $x\mapsto \psi_i(bxb^*)$ or  $x\mapsto d\psi_i(x)d^*$, for some $b\in \B$ or $d \in \D$,
		
		\noindent\hspace{.8cm} (2) a finite sum $\psi_{i_1}+\cdots\psi_{i_k}$ of the maps in the same family,
		
		\noindent\hspace{.8cm} (3) $x\mapsto \alpha^\D_t(\psi_i(x))$ or  $x\mapsto \psi_{t\cdot i}(\alpha^\B_t(x))$, for some $i\in I$ and $t\in\G$,
		
		\item[($iii$)] for $\G$-admissible families $\{\phi_i: \B_i\to \D\}_{i\in I}$ and $\{\psi_j: \D_j\to \mathcal C\}_{j\in J}$, the family consisting of compositions $\psi_j\circ\phi_i: \B_i\to\mathcal C$, for indices with $\phi_i(\B_i)\subseteq \D_j$,  is $\G$-admissible,
		
		\item[($iv$)] for $\G$-admissible families $\{\phi_i: \B_i\to \D\}_{i\in I}$ and $\{\psi_j: \B_j\to \D\}_{j\in J}$, their union is $\G$-admissible.
		\end{enumerate}
Similarly, the class of $\G$-coadmissible c.p. maps is characterized through the same set of rules, with admissible replaced by coadmissible,  $(i)$ and part (3) in $(ii)$ replaced by:
	
	\begin{enumerate}
		\item[($i$)$^{'}$] a $\G$-cofamily is $\G$-coadmissible,
		
		\item[(3)$^{'}$] $x\mapsto ({\rm id}_\D\otimes\varphi)\delta_\D\big(\psi_i(x)\big)$ or  $x\mapsto \psi_{\delta(i)}\big(({\rm id}_{\B_{\delta(i)}}\otimes\varphi)\delta_\B(x)\big)$, for some $i\in I$ and $\varphi\in C^*_r(\G)^*$.
	\end{enumerate}

\noindent When $\B$ and $\D$ are operator  $\A$-$\G$-(co)modules, a (co)family $\{\psi_i: \B_i\to \D\}_{i\in I}$ of c.p. maps is called (co)admissible if the (co)family is $\G$-(co)admissible, and all the maps $\psi_i$ are $\A$-admissible.
\end{defn}

We use the notation $F\ssubset X$ to say that $F$ is a finite subset of $X$.

\begin{defn} \label{e-g-nuc}
Let $\B$ and $\D$ be operator $\A$-$\G$-(co)modules and $\CE$ be an operator system with unit $1_\CE$ such that $\CE\otimes\mathbb B(\ell^2(\G^\infty))$ has an $\A$-$\G$-(co)module structure. Given $F\ssubset \G^\infty$, let $P_F: \ell^2(\G^\infty)\twoheadrightarrow
\ell^2(F)$ be the corresponding orthogonal projection. A c.c.p. $\A$-module and $\G$-(co)map $\theta: \B \rightarrow \D$ is called $\CE$-$\G$-nuclear, as a map between $\A$-$\G$-(co)modules, if for each $\varepsilon>0$ and $S\ssubset \B$, there is a (co)admissible c.c.p. map $\varphi: \B\rightarrow \CE\otimes\mathbb B(\ell^2(\G^\infty))$ and (co)admissible family $\{\psi_F: \M_F(\CE)\rightarrow \D\}_{F\ssubset \G^\infty}$ of c.c.p. maps such that,
$$\|\psi_F \varphi_F(b)-b\|\leq\varepsilon \ \ (b\in S),$$
where the c.c.p. (co)admissible map $\varphi_F: \B\rightarrow \M_F(\CE)$ is defined by
$$\varphi_F(b):=(1_\CE\otimes P_F)\varphi(b)(1_\CE\otimes P_F),$$ for $b\in\B$.
We say that $\B$ is $\CE$-$\G$-nuclear as an $\A$-$\G$-(co)module, if ${\rm id}_\B$ is $\CE$-$\G$-nuclear as an $\A$-module and $\G$-(co)map.
\end{defn}

\begin{rem}
$(i)$
Note that in the above definition, we have incorporated the $\G$-action not only in the definition of the class of $\G$-admissible maps, but also in the module structure of $\M_F(\CE)$. Also, each $\mathbb M_F(\CE)$ (and in particular, $\CE$ itself) is getting its $\A$-module structure as a subspace from  $ \CE\otimes\mathbb B(\ell^2(\G^\infty))$, but it is not necessarily a $\G$-(co)submodule of  $ \CE\otimes\mathbb B(\ell^2(\G^\infty))$. The latter fact explains why we do not directly require existence of
c.c.p. (co)admissible maps $\varphi_F: \B\rightarrow \M_F(\CE)$ and $\psi_F: \M_F(\CE)\rightarrow \D$ such that
$$\|\psi_F \varphi_F(b)-b\|\leq\varepsilon \ \ (b\in S).$$
When, each $\mathbb M_F(\CE)$ happen to be a $\G$-(co)submodule of  $ \CE\otimes\mathbb B(\ell^2(\G^\infty))$, these definitions are equivalent.

$(ii)$ When $\G$ is the trivial group, or more generally, when $\G$ (co)acts trivially on all the $\mathcal A$-modules involved, Definitions \ref{e-nuc} and \ref{e-g-nuc} are equivalent. Indeed, in this case, every family of c.p. maps is automatically a $\G$-(co)family. When  $\G$ is trivial and $\A=\mathbb C$, these definitions are equivalent to the definition of nuclearity for operator systems \cite[page 506]{eor}
\end{rem}

As in \cite[Lemma 2.3]{Ami}, we have the following result.

\begin{lem} \label{independence}
	Let $\B$ and $\D$ be operator $\A$-$\G$-(co)modules and $\CE$ be an operator system such that $ \CE\otimes\mathbb B(\ell^2(\G^\infty))$ has an $\A$-$\G$-(co)module structure. Let $\theta: \B \rightarrow \D$ be a c.c.p. $\A$-module and $\G$-(co)map which is $\CE$-$\G$-nuclear, as a map between $\A$-$\G$-(co)modules.
	
	$(i)$ $($restriction$)$ If $\C\subseteq \B$ is an $\A$-$\G$-sub(co)modules, then the restriction map $\theta|_\C: \C\to \D$ is $\CE$-$\G$-nuclear, as a map between $\A$-$\G$-(co)modules.
	
	$(ii)$ $($dependence on range$)$ If $\C\subseteq \D$ is an $\A$-$\G$-sub(co)modules with $\theta(\B)\subseteq \C$, then under any of the following conditions, $\theta: \B\to \C$  is $\CE$-$\G$-nuclear, as a map between $\A$-$\G$-(co)modules:
	
	\hspace{.3cm} $(1)$ there is a conditional expectation $\mathbb E: \D\to \C$,
	
	\hspace{.3cm} $(2)$ there is a sequence of c.c.p. $\A$-$\G$-(co)admissible maps $\mathbb E_n: \D\to \C$ such that $\mathbb E_n\to id_\C$ on $\C$ in the point-norm topology.
	
	$(ii)$ $($composition$)$ If $\C$ is an $\A$-$\G$-(co)module and $\sigma: \D\to \C$ and $\tau: \C\to \B$ are c.c.p. $\A$-module and $\G$-(co)maps, then the compositions $\sigma\circ\theta$ and $\theta\circ \tau$ are $\CE$-$\G$-nuclear, as maps between $\A$-$\G$-(co)modules.
\end{lem}

\section{Main Results}

In this section, $\G$ is a discrete group with identity element $e$, that acts on a unital C$^*$-algebra $\A$ with unit $1$, through an action $\alpha$. We identify $\M_n(\A)$ with $\A\otimes \M_n(\BC)$. Let $C_c(\G,\A)$ be the linear space of finitely supported functions on $\Gamma$ with values in $\A$. We denote elements of $C_c(\G,\A)$ by formal finite sums $\sum_{s\in\G} a_ss$, where $a_s\in\A$ is zero except for finitely many $s\in\G$. This should be interpreted as a function: $\G\to\A; s\mapsto a_s$. We denote the reduced crossed product of the C$^*$-dynamical system $(\A,\G,\alpha)$ by $\A\rtimes_{\alpha,r} \G$, or simply by $\A\rtimes_{r} \G$ when the action is known from the context, and identify $\A\subseteq \A\rtimes_{r} \G$ as well as $\G\subseteq \A\rtimes_{r} \G$ via canonical embeddings. We frequently use the facts that $\{as: a\in\A, s\in\G\}$ is a total set in $\A\rtimes_r\G$ and there is a canonical faithful conditional expectation $\mathbb E:\A\rtimes_{\alpha,r} \G \rightarrow \A$, sending $\sum_{s\in\G} a_ss$ to $a_e$ (see \cite{BO} for details).

Let $Z(\A)$ be the center of $\A$ with positive cone  $Z(\A)^+$. Let us  recall the notion of amenable actions on C$^*$-algebras  (cf., \cite[Definition 4.3.1]{BO}).

\begin{defn}\label{am}
An action $\alpha$ of a discrete group $\G$ on a unital C$^*$-algebra $\A$ is amenable if there exists a net $\{T_i\}_{i\in \BN}$ of finitely supported functions $T_i: \G\rightarrow Z(\A)^+$ such
that $\sum\limits_{s\in \G}T_i(s)^2=1$ and for each $t\in \G$,
$$\sum_{s\in \G}\big(T_i(s)-\alpha_t(T_i(t^{-1}s))\big)^*\big(T_i(s)-\alpha_t(T_i(t^{-1}s))\big)\rightarrow 0,$$
in norm, as $i\to\infty$.

\end{defn}

Since $\A$ is a C$^*$-subalgebra of $\A\rtimes_{r} \G$, the reduced crossed product $\A\rtimes_{r} \G$ could be thought as a C$^*$-module with module action given by the right multiplication map.  Hence, the right $\A$-module structure on $\A\rtimes_r\G$ is
$$as\cdot b:=asb=a\alpha_s(b)s, \ \ (a,b\in \A, s\in\G).$$

For a faithful representation $\A\hookrightarrow\mathbb B(\mathcal H)$, we define a new representation of $\A$ on $\mathcal H \otimes \ell^2(\G^\infty)$ by
$$\pi(a)(\xi\otimes \delta_s)=(\alpha_{\bar s^{-1}}(a)\xi)\otimes \delta_s, \ (a\in \A, s\in \G^\infty).$$
where with the notations of previous section, $\bar t:=t_1\cdots t_\ell\in\G$, for $t=(t_1,\cdots, t_\ell)\in\G^\infty$.
In fact, we have that $$\pi(a)=\sum_{s\in \G^\infty}\alpha_{\bar s^{-1}}(a)\otimes e_{s,s}, \ (a\in \A),$$
where $e_{s,t}$ are the matrix units of $\mathbb B(\ell^2(\G^\infty))$.
Also, compatible with what defined in the previous section, we define the right action of $\A$ on $\A\otimes \mathbb B(\ell^2(\G^\infty))$ by
$$(a\otimes e_{s,t})\cdot b:=(a\otimes e_{s,t})\pi(b)=a\alpha_{\bar t^{-1}}(b)\otimes e_{s,t}, \ (a,b\in \A, s,t\in \G^\infty).$$
These extend to actions on $\B:=\A\rtimes_r\G$ and $\D:=\A\otimes\mathbb B(\ell^2(\G^\infty))$, respectively.

Let $U$ be the unitary representation of $\G$ on $B(\ell^2(\G^\infty))$ such that
$$U_r(\delta_s)=\delta_{rs}, \ (r\in \G, s\in \G^\infty).$$
In fact, we have that $$U_r=\sum_{s\in \G^\infty}e_{rs,s},$$
for all $r\in \G$.
Next, we define the $\G$-actions on these $\A$-modules by
$$r\cdot(as):=rasr^{-1}=\alpha_r(a)rsr^{-1},  \ \ (a\in \A, s,r\in\G).$$
and
$$r\cdot(a\otimes e_{s,t}):=U_r(a\otimes e_{s,t})U^*_r=a \otimes e_{rs,rt}, \ (a\in \A, r\in\G, s,t\in \G^\infty).$$

It is easy to see that these maps are homomorphism.
Let us next observe that these are $\G$-actions compatible with the module structure in both cases. For  $a,b\in \A$ and $s,r\in \G$,
\begin{align*}
	r\cdot (as\cdot b)&=r\cdot(a\alpha_s(b)s)=\alpha_r(a)\alpha_{rs}(b)rsr^{-1}\\
	&=\alpha_r(a)\alpha_{rsr^{-1}}(\alpha_r(b))rsr^{-1}=\big(r\cdot (as)\big)\cdot \alpha_r(b).
\end{align*}

Similarly, for  $a,b\in \A$ and $s,t\in \G^\infty, r\in \G$,
\begin{align*}
	r\cdot \big((a\otimes e_{s,t})\cdot b\big)&=r\cdot (a\alpha_{\bar t^{-1}}(b)\otimes e_{s,t}) =a\alpha_{\bar t^{-1}}(b)\otimes e_{rs,rt}\\
	&=a\alpha_{\bar t^{-1}r^{-1}}(\alpha_r(b))\otimes e_{rs,rt}=\big(r\cdot (a\otimes e_{s,t})\big)\cdot \alpha_r(b).
\end{align*}
These extend to  well defined compatible $\A$-module actions and $\G$-actions on $\A\rtimes_r\G$, and $\A\otimes \mathbb B(\ell^2(\G^\infty))$, turning them into operator $\A$-$\G$-modules.

\begin{thm}\label{thm:act-am-a-nu}
	If the action $\alpha$ is amenable, then the reduced C$^*$-crossed product $\A\rtimes_{r} \G$ is $\A$-$\G$-nuclear as an $\A$-$\G$-module.
\end{thm}

\begin{proof}
	Let $\A$ be represented faithfully in $\CH$, and $\varphi$ be the regular representation of $\A\rtimes_{r} \G$ in $\CH\otimes \ell^2(\G)$, which is defined on $C_c(\G,\A)$ by, 	$$\varphi\big(\sum_{s\in\G} a_ss\big)=\sum_{s,t\in \G}\alpha_{t^{-1}}(a_s)\otimes e_{t,s^{-1}t},$$
	with only finitely many $a_s$ being nonzero. For $F\ssubset \G^\infty$, consider the map $\psi_F:\mathbb M_F(\A) \rightarrow C_c(\G, \A)$, defined by,
	$$\psi_F(\sum_{s,t\in F}a_{s,t}\otimes e_{s,t})=\sum_{s,t\in F}\alpha_{\bar s}(a_{s,t}){\bar s}{\bar t}^{-1},$$
Then $\varphi$ extends to a $*$-homomorphism on $\A\rtimes_{r}\G$, and it follows, by an argument as in the proof of \cite[Lemma 4.2.3]{BO}, that $\psi_F$ is a c.c.p. map.
	
	Next, let us observe that these are  $\A$-module maps. For  $a,b \in \A$, and $s,t\in \G$,
	\begin{align*}
		\varphi(as\cdot b)&=\varphi(a\alpha_s(b)s)=\sum_{t\in \G} \alpha_{t^{-1}}(a\alpha_s(b))\otimes e_{t,s^{-1}t}\\
		&=\sum_{t\in \G} \alpha_{t^{-1}}(a)\alpha_{t^{-1}s}(b)\otimes e_{t,s^{-1}t}
		=\varphi(as)\cdot b.
	\end{align*}
	Similarly, for $a,b \in \A$, and $s,t\in F$,
	\begin{align*}
		\psi_F((a\otimes e_{s,t})\cdot b)&=\psi_F(a\alpha_{\bar t^{-1}}(b)\otimes e_{s,t})=\alpha_{\bar s}(a)\alpha_{\bar s\bar t^{-1}}(b)\bar s\bar t^{-1}=\psi_F(a\otimes e_{s,t})\cdot b.
	\end{align*}
	Next, let us observe that $\varphi$ is a $\G$-map and $\{\psi_F\}$ is a $\G$-family.
	For  $a \in \A$, and $r, s,t\in \G$,
	\begin{align*}
		\varphi(r\cdot as)&=\varphi(\alpha_r(a)rsr^{-1})=\sum_{t\in \Gamma} \alpha_{t^{-1}}(\alpha_r(a))\otimes e_{t,rs^{-1}r^{-1}t}\\
		&=\sum_{t\in \Gamma} \alpha_{t^{-1}}(a)\otimes e_{rt,rs^{-1}t}=\sum_{t\in \Gamma} r\cdot(\alpha_{t^{-1}}(a)\otimes e_{t,s^{-1}t})=r\cdot\varphi(as).
	\end{align*}
	For  $a \in \A$, and $r\in\G$, and $s,t\in F$,
	\begin{align*}
		\psi_{rF}(r\cdot (a\otimes e_{s,t}))&=\psi_{rF}(a\otimes e_{rs,rt})=\alpha_{r\bar s}(a)r\bar s\bar t^{-1}r^{-1}\\&=r\cdot(\alpha_{\bar s}(a)\bar s\bar t^{-1})=r\cdot\psi_F(a\otimes e_{s,t}).
	\end{align*}
	Let $F_i$ be the support of $T_i$, for $i\geq 1$. Let $P_{F_i}$ be the finite rank projection onto the span of $\{\delta_g:g\in F_i\}$, then $1\otimes P_{F_i}$  is in $\A\otimes \mathbb B(\ell^2(\Gamma))$. For the c.c.p admissible maps $\varphi_i= (1\otimes P_{F_i})\circ \varphi\circ (1\otimes P_{F_i})$, we have $\varphi_i(\A\rtimes_{r} \G) \subseteq \M_{F_i}(\A)$, for each $i$.

	For the selfadjoint element,
	$$X_i=\sum_{t\in F_i} \alpha_{t^{-1}}(T_i(t))\otimes e_{t,t},$$ and  the c.p. compression $\theta_i$ by $X_i$, put  $\psi_i:=\frac{1}{|F_i|}\psi_{F_i}\circ\theta_i: \M_{F_i}(\A)\to\A\rtimes_r\G$, it follows from \cite[Lemma 4.3.2]{BO} and \cite[Lemma 4.3.3]{BO} that,
	$$\|\psi_i\circ \varphi_i(x)-x\|\rightarrow 0,$$
	for each $x\in \A\rtimes_{r} \G$. Therefore, $\A\rtimes_{r} \G$ is $\A$-$\G$-nuclear as an $\A$-$\G$-module.
\end{proof}

Next let us treat the case of coactions. The reduced crossed product $\A\rtimes_{r} \G$ also carries a  canonical reduced $\G$-coaction, given  by
$$\delta(as):=as\otimes \lambda_s \ \ (a\in \A, s\in\G).$$
Also $\G$ naturally coacts on $\mathbb M_F(\A)$ by
$$\delta_F(a\otimes e_{s,t}):=(a\otimes \lambda_{st^{-1}})\otimes e_{s,t}\ \ (a\in \A, s,t\in F),$$
where we have canonically identified $\mathbb M_F(\A)\otimes C^*_r(\G)$ with $\mathbb M_F\big(\A\otimes C^*_r(\G)\big)$. Note that both coactions restricted to $\A$, identified respectively with $\{a1: a\in\A\}$ and $\mathbb M_{\{1\}}(\A)$ boil down to a trivial restricted coaction,
$$\delta_\A: \A\to \A\otimes C^*_r(\G); \ \ a\mapsto a\otimes \lambda_1 \ \ (a\in\A),$$
which means that we are not pre-assuming any $\G$-comudule structure on $\A$. The same holds for full coactions, and there are canonical full $\G$-coaction on these two algebras, given the same, but  with $\lambda_s$ replaced by $\omega_s$, where  $\lambda: \G\to \mathcal U(\ell^2(\G))$ is the left regular representation and  $\omega: \G\to\mathcal U(\mathcal H_u); \ \omega:=\oplus_{\sigma\in\hat\G}\sigma$, is the universal representation on the universal Hilbert space $\mathcal H_u:=\bigoplus_{\sigma\in\hat \G} \mathcal H_\sigma$.

Both in the full and reduced cases, these extend to  well defined $\G$-coactions on $\A\rtimes_r\G$ and $\M_F(\A)$, turning them into operator $\A$-$\G$-comodules. For instance in the reduced case, after natural identifications, for $r,s,t\in\G$ and $a,b\in \A$,
\begin{align*}
	\delta(as\cdot b)&=\delta(a\alpha_s(b)s)=a\alpha_{s}(b)s\otimes\lambda_s=(as\cdot b)\otimes \lambda_s\\&=(as\otimes \lambda_s)\cdot(b\otimes\lambda_1)=\delta(as)\cdot \delta_\A(b),
\end{align*}
and
\begin{align*}
	\delta_F\big((a\otimes e_{s,t})\cdot b\big)&=\delta_F\big(a\alpha_{t^{-1}}(b)\otimes e_{s,t}\big)=(a\alpha_{t^{-1}}(b)\otimes\lambda_{st^{-1}})\otimes e_{s,t}\\&=[(a\otimes\lambda_{st^{-1}})\cdot(b\otimes\lambda_1)]\otimes e_{s,t}=[(a\otimes\lambda_{st^{-1}})\otimes e_{s,t}]\cdot(b\otimes\lambda_1)\\&=\delta_F(a\otimes e_{s,t})\cdot\delta_\A(b).
\end{align*}
These maps are also homomorphism: for $r,s,t\in\G$ and $a,b\in \A$,
\begin{align*}
	\delta(asbt)&=\delta(a\alpha_s(b)st)=a\alpha_{s}(b)st\otimes\lambda_{st}=(as\otimes \lambda_s)(bt\otimes\lambda_t)\\&=\delta(as)\delta(bt),
\end{align*}
also,
\begin{align*}
	\delta_F\big((a\otimes e_{s,t})(b\otimes e_{u,v})\big)&=\delta_F(ab\otimes e_{s,v})=(ab\otimes\lambda_{sv^{-1}})\otimes e_{s,v}\\&=[(a\otimes\lambda_{st^{-1}})\otimes e_{s,t}][(b\otimes\lambda_{uv^{-1}})\otimes e_{u,v}]\\&=\delta_F(a\otimes e_{s,t})\delta_F(b\otimes e_{u,v}),
\end{align*}
when $t=u$, and all the terms are zero, otherwise.

Now we are ready to prove the same result for the comodule structure.

\begin{thm}\label{thm:coact-am-a-nu}
	If the action $\alpha$ is amenable, then the reduced C$^*$-crossed product $\A\rtimes_{r} \G$ is $\A$-$\G$-nuclear as an $\A$-$\G$-comodule.
\end{thm}

\begin{proof}
	As in the proof of Theorem \ref{thm:act-am-a-nu}, let $F_i$ be the support of $T_i$, for $i\geq 1$, and consider the  c.c.p. module maps $\varphi_i: \A\rtimes_{r} G\rightarrow M_{F_i}(\A)$ and $\psi_i: M_{F_i}(\A)\rightarrow \A\rtimes_{r} G$.
	Let us observe that these are also $\G$-comaps, with respect to the coactions $\delta$ and $\delta_F$ defined above.
	After canonical identifications, for  $a \in \A$ and $s,r\in \G$,  and $i\geq 1$,
	\begin{align*}
		(\varphi_i\otimes{\rm id})\delta(as)&=(\varphi_i\otimes{\rm id})(as\otimes \lambda_s)\\&=\sum_{t\in F_i\cap sF_i} (\alpha_{t^{-1}}(a)\otimes e_{t,s^{-1}t})\otimes \lambda_s\\&=\sum_{t\in F_i\cap sF_i} (\alpha_{t^{-1}}(a)\otimes \lambda_s)\otimes e_{t,s^{-1}t}\\&=\sum_{t\in F_i\cap sF_i} (\alpha_{t^{-1}}(a)\otimes \lambda_{tt^{-1}s})\otimes e_{t,s^{-1}t}\\&=\sum_{t\in F_i\cap sF_i} \delta_F\big(\alpha_{t^{-1}}(a)\otimes e_{t,s^{-1}t}\big)
		=\delta_F\varphi_i(as),
	\end{align*}
	and
	\begin{align*}
		(\psi_i\otimes{\rm id})\delta_F(a\otimes e_{s,t})&=(\psi_i\otimes{\rm id})\big((a\otimes\lambda_{st^{-1}})\otimes e_{s,t}\big)\\&=(\psi_i\otimes{\rm id})\big((a\otimes e_{s,t})\otimes\lambda_{st^{-1}}\big)
		\\&=\frac{1}{|F_i|}(\alpha_s(a)st^{-1})\otimes\lambda_{st^{-1}}\\&=\frac{1}{|F_i|}\delta(\alpha_s(a)st^{-1})=\delta\psi_i(a\otimes e_{s,t}),
	\end{align*}
	where the third (the second) equality of the first (the second) calculation is due to the fact that we have  identified $\mathbb M_{F_i}(\A)\otimes C^*_r(\G)$ canonically with $\mathbb M_{F_i}(\A\otimes C^*_r(\G))$. Now compression by selfadjoint element $X_i:=\sum_{t\in F_i} \alpha_{t^{-1}}(T_i(t))\otimes e_{t,t}$  is a coadmissible c.p map $\theta_i$, and as before,
	$$\psi_i\theta_i \varphi_i(x)-x\rightarrow 0,$$ in norm, as $i\to\infty$,
	for each $x\in \A\rtimes_{r} \G$. Therefore, $\A\rtimes_{r} \G$ is $\A$-$\G$-nuclear as a $\A$-$\G$-comodule.
\end{proof}

Since any action of an amenable group on a unital C$^*$-algebra is amenable, we have the following corollary.

\begin{cor}\label{cor:gr-am-nu}
If $\G$ is amenable, then $\A\rtimes_{r} \G$ is $\A$-$\G$-nuclear both as an $\A$-$\G$-module and an $\A$-$\G$-comodule.
\end{cor}

From now on, we call $\A\rtimes_{r} \G$ is \emph{$\A$-$\G$-nuclear} if $\A\rtimes_{r} \G$ is $\A$-$\G$-nuclear as an $\A$-$\G$-module or as an $\A$-$\G$-comodule.

\begin{lem}\label{lem:con-exp-tra}
Let $\rho$ be an $\alpha$-invariant state on $\A$ and $\rho'=\rho\circ \mathbb E$. For every $x\in \A\rtimes_{r} \G$ and $t\in \G$, we
have $\rho'(xt)=\rho'(tx)$.
\end{lem}

\begin{proof}
For  $x:=\sum\limits_{s\in \G} a_{s}s\in C_c(\G,\A)$ and $t \in \G$, we have,
\begin{align*}
\rho'(xt)
&=\rho'((\sum\limits_{s\in G} a_{s}s)t)
=\rho( a_{t^{-1}})=\rho(\alpha_{t}(a_{t^{-1}}))\\
&=\rho'(\sum\limits_{s\in G} \alpha_{t}(a_{s})ts)=\rho'(tx).
\end{align*}
By a density argument, the same holds for all $x\in \A\rtimes_{r} \G$.
\end{proof}

\begin{thm}\label{thm:al-nu-am}
Suppose that $\A$ is nuclear and has an $\alpha$-invariant state $\rho$. If the reduced crossed product $\A\rtimes_{r} \G$ is $\A$-$\G$-nuclear, then $\G$ is amenable.
\end{thm}

\begin{proof}
Since $\A$ is nuclear, \cite[Proposition 10.1.7]{BO} shows that $\M_F(\A)$ is nuclear, for each $F\ssubset \G^\infty$. It follows from \cite[Exercie 2.3.11]{BO} that $\A\rtimes_{r} \G$ is nuclear. Hence there exists a net $\{\Phi_n\}_{n\in I}$ of finite rank u.c.p. maps from $\A\rtimes_{r} \G$ to itself converge to the identity map in the point norm topology. By the argument of \cite[Theorem 4.3]{MSTT}, we can assume that the range of $\Phi_n$ is in $C_c(\G, \A)$.
For each $n$, let
\begin{equation*}
\omega_n(s)=\rho'(\Phi_n(s)s^{-1})
\end{equation*}
for all $s \in \G$. For all $s_1,\ldots, s_m\in \G$ and $c_1,\ldots, c_m\in \mathbb C$, the positivity of $\rho'$ yields
\begin{align*}
\sum^m_{i,j=1}c_i\bar{c_j}\omega_n(s^{-1}_js_i)
&=\sum^m_{i,j=1}c_i\bar{c_j}\rho'((\Phi_n(s^{-1}_js_i))(s^{-1}_is_j))\\
&=\sum^m_{i,j=1}\rho'(\bar{c_j}s_j\Phi_n(s_j^{-1}s_i)c_is_i^{-1})\geq 0.
\end{align*}
Hence, $\omega_n$ is positive definite on $\G$. Moreover, we have
\begin{align*}
|\omega_n(s)-1|&=|\rho'(\Phi_n(s)s^{-1})-1|=|\rho'(\Phi_n(s)s^{-1})-\rho'(ss^{-1})|\\
&=|\rho'((\Phi_n(s)-s)s^{-1})|\leq\|\Phi_n(s)-s)\|\rightarrow 0,
\end{align*}
for all $s\in \G$. Since $\Phi_n$ is finite rank, $\omega_n$ has finite support.
\end{proof}

It follows from Corollary \ref{cor:gr-am-nu} and Theorem \ref{thm:al-nu-am} that if $\A$ is nuclear and has an $\alpha$-invariant state $\rho$, then $\G$ is amenable if and only if $\A\rtimes_{r} \G$ is $\A$-$\G$-nuclear. Let $\gamma$ act on $C_b(\G)$ by conjugation, then there exists a $\gamma$-invariant tracial state $\tau_e$ on $C_b(\G)$, such that
$\tau_e(f)=f(e)$,
for all $f\in C_b(\G)$. Hence, we get the following characterization of amenability of $\G$.

\begin{prop}\label{prop:nu-am-eq}
$\G$ is amenable if and only if $C_b(\G)\rtimes_r \G$ is $C_b(\G)$-$\G$-nuclear.
\end{prop}

\begin{ex}
(a) Let $\BZ$ act on a unital C$^*$-algebra $\A$ through an action $\alpha$. Since $\BZ$ is amenable, it follows from Corollary \ref{cor:gr-am-nu} that $\A\rtimes_{r} \BZ$ is $\A$-$\BZ$-nuclear.

(b) Let $\partial \mathbb F_2$ be the ideal boundary of $\mathbb F_2$. It is known that $\mathbb F_2$ acts amenably on $C(\partial \mathbb F_2)$ through an action induced by left multiplication (see \cite{BO}).
It follows from Theorems \ref{thm:act-am-a-nu} and \ref{thm:coact-am-a-nu} that $C(\partial \mathbb F_2)\rtimes_{r} \mathbb F_2$ is $C(\partial \mathbb F_2)$-$\mathbb F_2$-nuclear.

(c) It follows from \cite{BHV} that the special linear group $SL_3(\mathbb Z)$ has property $(T)$. Moreover, a discrete group $\G$ is finite if and only if $\G$ is amenable and has property $(T)$. Hence, Proposition \ref{prop:nu-am-eq} shows that $C_b(SL_3(\mathbb Z))\rtimes_{r} SL_3(\mathbb Z)$ is not $C_b(SL_3(\mathbb Z))$-$SL_3(\BZ)$-nuclear.
\end{ex}

In \cite{Haag}, Haagerup introduced the completely bounded approximation property for C$^*$-algebras and an important isomorphism invariant $\Lambda(\A)$ for a C$^*$-algebras $\A$.
\begin{defn}
We say a C$^*$-algebra $\A$ has the completely bounded approximation property (CBAP) if there exist a constant $C>0$ and a net of finite rank c.b. maps $\Phi_i: \A\rightarrow \A$ such that
$$\|\Phi_i(a)-a\|\rightarrow 0$$
for all $a\in \A$ and $\sup \{\|\Phi_i\|_{cb}\}\leq C$.
\end{defn}
The Haagerup constant $\Lambda(\A)$ is the infimum of all $C$ for which such a net $\{\Phi_i\}$ exists. We set $\Lambda(\A)=\infty$ if $\A$ does not have the CBAP.

\begin{lem}\label{lem:app-cbap}
Let $\{\A_i\}_{i\in I}$ be a net of C$^*$-algebras. Assume that $\varphi_i :\A\rightarrow \A_i$ and $\psi_i :\A_i\rightarrow \A$ are c.b. maps such that $\psi_i\circ\varphi_i\rightarrow {\rm id}_\A$ in the point-norm topology, and
$$\sup\|\varphi_i\|_{cb}\leq m_1, \qquad \sup\|\psi_i\|_{cb}\leq m_2.$$ If each $\A_i$ has the CBAP and $\Lambda(\A_i)\leq M$, then $\A$ has the CBAP.
 In fact, $$\Lambda(\A)\leq m_1 m_2 M.$$
\end{lem}

\begin{proof}
Let $F\ssubset \A$ and $\varepsilon> 0$ be given. There exists $i\in I$ such that
\begin{align*}
\|\psi_i\circ \varphi_i(a)-a\| \leq \frac{\varepsilon}{2},
\end{align*}
for all $a\in F$. By assumption, there exists a c.b. map $\Phi: \A_i\rightarrow \A_i$ such that  $\|\Phi\|_{cb}\leq M$,
and
\begin{align*}
\|\Phi(\varphi_i(a))-\varphi_i(a)\|\leq \frac{\varepsilon}{2 m_2}
\end{align*}
for all $x\in F$. Let
$$\Phi_j=\psi_i\circ \Phi \circ \varphi_i,$$
where $j=(F,\varepsilon)$. Then $\Phi_j$ is a c.b. map on $\A$ with $\|\Phi_j\|_{cb}\leq m_1 m_2 M$.
Moreover, we get
\begin{align*}
\|\Phi_j(a)-a\|&\leq \|\psi_i\circ \Phi \circ \varphi_i(a)-\psi_i\circ \varphi_i(a)\|+\|\psi_i\circ \varphi_i(a)-a\|\\
&\leq m_2\|\Phi(\varphi_i(x))-\varphi_i(x)\|+\frac{\varepsilon}{2}\leq \varepsilon
\end{align*}
for all $a\in F$.
Hence, the net $\{\Phi_j\}_{j\in J}$ indexed by $J=\{(F,\varepsilon)\mid F\ssubset \A, \varepsilon> 0\}$ shows that $\A$ has the CBAP and $\Lambda(\A)\leq m_1 m_2 M$.
\end{proof}

\begin{lem}\label{lem:cr-pr-con-ex}
If $\A\rtimes_{r} \G$ has the CBAP, the $\A$ has the CBAP. In fact, $$\Lambda(\A) \leq \Lambda(\A\rtimes_{r} \G).$$
\end{lem}

\begin{proof}
Suppose that there exists a net $\{\Phi_i\}_{i\in I}$ of c.b. maps on $\A\rtimes_{r} \G$ witnessing the CBAP of $\A\rtimes_{r} \G$. Then $\{\mathbb E\circ \Phi_i\}_{i\in I}$ witnesses the CBAP of $\A$. Hence, we have $\Lambda(\A)\leq \Lambda( \A\rtimes_{r} \G).$
\end{proof}

\begin{thm}\label{thm:a-nuc-cbap}
If $\A\rtimes_{r} \G$ is $\A$-$\G$-nuclear and $\A$ has the CBAP, then $\A\rtimes_{r} \G$ has the CBAP. In fact, $$\Lambda(\A\rtimes_{r} \G)=\Lambda(\A).$$
\end{thm}

\begin{proof}
For each  $F\ssubset \G^\infty$,
$$\Lambda(M_F(\A))=\Lambda(\A).$$
The $\A$-$\G$-nuclearity and Lemma \ref{lem:app-cbap} show that, $$\Lambda(\A\rtimes_{r} \G) \leq \Lambda(\A).$$
On the other hand, it follows from Lemma \ref{lem:cr-pr-con-ex} that $\Lambda(\A) \leq \Lambda( \A\rtimes_{r} \G)$.
Hence, $\A\rtimes_{r} \G$ has the CBAP, and $\Lambda(\A)=\Lambda(\A\rtimes_{r} \G)$.
\end{proof}

Since $\A\rtimes_{r} \G$ is $\A$-$\G$-nuclear whenever the action of $\G$ on $\A$ is amenable, the above theorem is a generalization of \cite[Theorem 3.4]{SiSm} and \cite[Theorem 2.2]{Meng}.
Using a similar argument, we also have the following result, which is a generalization of \cite[Theorem 4.2.6]{BO}.

\begin{thm}\label{thm:a-nuc-cbap}
Suppose that $\A\rtimes_{r} \G$ is $\A$-$\G$-nuclear, then $\A$ is exact if and only if $\A\rtimes_{r} \G$ is exact.
\end{thm}

\begin{proof}
This follows from \cite[Exercise 2.3.12]{BO} and \cite[Proposition 10.2.7]{BO}.
\end{proof}

\begin{rem}
If $\A\rtimes_{r} \G$ is $\A$-$\G$-nuclear and $\A$ is exact, then it follows from \cite[Proposition 10.2.3]{BO} that the group $\G$ is exact.
\end{rem}

We conclude this article with the following results on property $(T)$. In the case of possibly uncountable group, we also have the corresponding
statements for strong property $(T)$.

\begin{prop}\label{prop:a-nuc-T}
Suppose that $\G$ is countable, $\A$ is nuclear, and there exists an $\alpha$-invariant tracial state on $\A$, then the following statements are equivalent.
\begin{enumerate}[\ \ (1)]
\item $\G$ is finite.

\item $\A\rtimes_{r} \G$ is $\A$-$\G$-nuclear and $(\A\rtimes_{r} \G, C^*_r(\G))$ has (strong) property $(T)$.
\end{enumerate}
\end{prop}

\begin{proof}
(1)$\Rightarrow$(2) If $\G$ is finite, then $C^*_r(\G)$ has strong property $(T)$ (see \cite{LN}).  Hence, $(\A\rtimes_{r} \G, C^*_r(\G))$ has strong property $(T)$. Moreover, it follows from Corollary \ref{cor:gr-am-nu} that $\A\rtimes_{r} \G$ is $\A$-$\G$-nuclear.

(2)$\Rightarrow$(1) Since there exists an $\alpha$-invariant tracial state on $\A$, it follows from \cite[Proposition 3.4]{MN20} that $\G$ has property $(T)$. Moreover, Theorem \ref{thm:al-nu-am} shows that $\G$ is amenable. Hence, $\G$ is finite.
\end{proof}

Note that if $\G$ is finite, then $C_b(\G)\rtimes_{r} \G$  has strong property $(T)$. Therefore, we have a characterization of finiteness for countable discrete groups.

\begin{cor}\label{cor:a-nuc-T}
Suppose that $\G$ is countable, then $\G$ is finite if and only if $C_b(\G)\rtimes_{r} \G$ is $C_b(\G)$-$\G$-nuclear and has (strong) property $(T)$.
\end{cor}

\section*{Acknowledgement}
The first author was partially supported by Iran Nation Science Foundation (INSF GRant No. 4027151). The second author was supported by the Natural Science Foundation of Shandong Province (No. ZR2020MA008).

\end{document}